\documentclass[11pt, reqno]{amsart}
\usepackage{amsmath, amsthm, amscd, amsfonts, amssymb, graphicx, color}
\usepackage[bookmarksnumbered, colorlinks, plainpages]{hyperref}
\hypersetup{colorlinks=true,linkcolor=red, anchorcolor=green, citecolor=cyan, urlcolor=red, filecolor=magenta, pdftoolbar=true}
\newtheorem{theorem}{Theorem}[section]
\newtheorem{lemma}[theorem]{Lemma}
\newtheorem{proposition}[theorem]{Proposition}
\newtheorem{corollary}[theorem]{Corollary}
\theoremstyle{definition}
\newtheorem{definition}[theorem]{Definition}
\newtheorem{example}[theorem]{Example}

\theoremstyle{remark}
\newtheorem{remark}[theorem]{Remark}
\numberwithin{equation}{section}

\begin{document}

\setcounter{page}{1}

\title[Dynamical Properties and Some Classes of Non-porous Subsets]{Dynamical Properties and Some  Classes of Non-porous Subsets of Lebesgue Spaces}

\author[S. Ivkovi\'c]{Stefan Ivkovi\'c}

\address{Mathematical Institute of the Serbian Academy of Sciences and Arts,
	p.p. 367, Kneza Mihaila 36, 11000 Beograd, Serbia.}
\email{\textcolor[rgb]{0.00,0.00,0.84}{stefan.iv10@outlook.com}}

\author[S. \"Oztop]{Serap \"Oztop}
\address{Department of Mathematics, Faculty of Science, Istanbul University, Istanbul, Turkey}
\email{\textcolor[rgb]{0.00,0.00,0.84}{oztops@istanbul.edu.tr}}

\author[S.M. Tabatabaie]{Seyyed Mohammad Tabatabaie$^*$}

\address{Department of Mathematics, University of Qom, Qom, Iran.}
\email{\textcolor[rgb]{0.00,0.00,0.84}{sm.tabatabaie@qom.ac.ir}}

\address{
\newline
}

\subjclass[2010]{47A16, 28A05, 43A15, 43A62.
	\\
\indent $^{*}$Corresponding author}

\begin{abstract}
In this paper, we introduce several classes of non-$\sigma$-porous subsets of a general Lebesgue space. Also,  we study some linear dynamics of operators and show that the set of all non-hypercyclic vectors of a sequences of weighted translation operators on $L^p$-spaces is not $\sigma$-porous.
\end{abstract} 
\maketitle
\section{Introduction}
$\sigma$-porous sets, as a collection of very thin subsets of  metric spaces, were introduced and studied first time in \cite{dol} through a research on boundary behavior of functions, and then were applied in differentiation and Banach spaces theories in \cite{bel,pre}. The concepts  related to porosity have been active topics in recent decades because they can be adapted for many known notions in several kind of metric spaces; see the monograph \cite{zaj2}.  $\sigma$-porous subsets of $\mathbb{R}$ are null and of first category, while in every complete metric space without any isolated points these two categories are different \cite{zaj}. 
On the other hand, linear dynamics including hypercyclicity in operator theory received attention during the last years; see books \cite{bmbook,gpbook} and for instance \cite{cot,tabsaw,tabiv2}.
Recently, F. Bayart in \cite{bay} through study of hypercyclic shifts (which was previously studied in \cite{sa95}; see also \cite{ge00}) proved that the set of non-hypercyclic vectors of some classes of weighted shift operators on $\ell^2(\mathbb{Z})$ is a non-$\sigma$-porous set. This would be a new example of a first category set which is not $\sigma$-porous. In this work, by some idea from the proof of \cite[Theorem1]{bay} first we introduce a class of non-$\sigma$-porous subsets of general Lebesgue spaces, and then we develop the main result of \cite{bay} to sequences of weighted translation operators on general Lebesgue spaces in the context of discrete groups and hypergroups. In particular, we prove that if $p\geq 1$, $K$ is a discrete hypergroup, $(a_n)$ is a sequence with distinct terms in $K$, and   $w:K\rightarrow(0,\infty)$ is a bounded measurable function such that
	$$\sum_{n\in\mathbb{N}}\frac{1}{w(a_0)w(a_1)\ldots w(a_n)}\chi_{\{a_{n+1}\}}\in L^p(K),$$
	then the set of all non-hypercyclic vectors of the sequence $(\Lambda_n)_n$ is not $\sigma$-porous, where the operators $\Lambda_n$ are given in Definition \ref{defl}.
Also, we study non-$\sigma$-porosity  of non-hypercyclic vectors of weighted composition operators on $L^\infty(\Omega)$ for a general measure space $\Omega$ equipped with a nonnegative Radon measure and on $L^p(\mathbb{R},\tau)$, where $\tau$ is the Lebesgue measure on $\mathbb{R}$. We 	show that if   $G$ is a locally compact group, $\mu$ is a left Haar measure on $G$, $a\in G$, and $w:G\to(0,\infty)$ be a weight such that
$$\big(\frac{1}{w(a)w(a^2)\ldots w(a^{n})}\big)_n\in L^\infty(G,\mu),$$
then the set of all non-hypercyclic vectors of the weighted translation operator $T_{a,w,\infty}$ on $L^\infty(G,\mu)$ is not $\sigma$-porous.
\section{Non-$\sigma$-porous subsets of Lebesgue spaces}
In this section, we will introduce some classes of non-$\sigma$-porous subsets of Lebesgue spaces related to a fixed function. First, we recall the definition of the main notion of this paper. 
\begin{definition}
	Let $0<\lambda<1$. A subset $E$ of a metric space $X$ is called \emph{$\lambda$-porous} at $x\in E$ if for each $\delta>0$ there is an element $y\in B(x;\delta)\setminus\{x\}$ such that
	$$B(y;\lambda\,d(x,y))\cap E=\varnothing.$$
	$E$ is called \emph{$\lambda$-porous} if it is $\lambda$-porous at every element of $E$. Also, $E$ is called \emph{$\sigma$-$\lambda$-porous} if it is a countable union of $\lambda$-porous subsets of $X$.
\end{definition}
The following lemma plays a key role in the proof of main results of this section. This fact is a special case of \cite[Lemma2]{zaj1}; see also \cite[Lemma2]{bay}.
\begin{lemma}\label{lem1}
	Let $\mathcal F$ be a non-empty family of non-empty closed subsets of a complete metric space $X$ such that for each $F\in\mathcal F$ and each $x\in X$ and $r>0$ with $B(x;r)\cap F\neq \varnothing$, there exists an element $J\in\mathcal F$ such that 
	$$\varnothing\neq J\cap B(x;r)\subseteq F\cap B(x;r)$$
	and $F\cap B(x;r)$ is not $\lambda$-porous at all elements of $J\cap B(x;r)$. Then, every set in $\mathcal F$ is not $\sigma$-$\lambda$-porous.
\end{lemma}
The next result is a development of of \cite[Theorem1]{bay}. Same as \cite{bay}, the proof of this theorem is based on Lemma \ref{lem1}.
\begin{theorem}\label{thm1}
		Let $p\geq 1$, $\Omega$ be a locally compact Hausdorff space, $\mu$ be a nonnegative Radon measure on $\Omega$, and $A\subseteq \Omega$ be a Borel set such that
	\begin{equation}\label{cond1}
	|f|\chi_A\leq \|f\|_p \text{ a.e. }\qquad (f\in L^p(\Omega,\mu)).
	\end{equation}
	Then, for each measurable function $g$ on $\Omega$ with $g\chi_A\in L^p(\Omega,\mu)$, the set 
	$$\Gamma_{g}:=\big\{f\in L^p(\Omega,\mu):\,|f|\geq |g|\chi_A\,\,\,{\rm a.e.}\big\}$$
	is not $\sigma$-porous in $L^p(\Omega,\mu)$. 
\end{theorem}
\begin{proof}
Fix an arbitrary number $0<\lambda\leq \frac{1}{2}$, and pick $0<\beta<\lambda$.	Denote
	$$\mathcal F:=\big\{\Gamma_{g}:\,\, g\chi_A\in L^p(\Omega,\mu)\big\}.$$
We will show that the collection $\mathcal F$ satisfies the conditions of Lemma \ref{lem1}. Let $g\in L^p(\Omega,\mu)$. Without lossing the generality, we can assume that $g$ is a nonnegative function. Trivially, $\Gamma_{g}\neq \varnothing$. Let $(f_n)$ be a sequence in $\Gamma_{g}$ and $f_n\rightarrow f$ in $L^p(\Omega,\mu)$. Then, by \eqref{cond1}, $|f|\geq g\chi_A$ a.e., and so $f\in \Gamma_{g}$. Therefore, every element of the collection $\mathcal F$ is a closed subset of $L^p(\Omega,\mu)$. Now, assume that $f\in L^p(\Omega,\mu)$ and $r>0$ with $B(f;r)\cap \Gamma_{g}\neq\varnothing$. We find a measurable function $h$ with $0\leq h\chi_A\in L^p(\Omega,\mu)$ such that 
$$\varnothing\neq B(f;r)\cap \Gamma_{h}\subseteq B(f;r)\cap \Gamma_{g},$$
and $B(f;r)\cap \Gamma_{g}$ is not $\lambda$-porous at elements of $B(f;r)\cap \Gamma_{h}$.
\par Since $\big(|f|+\beta^{-1}g\chi_A\big)^p\in L^1(\Omega,\mu)$ and $\mu$ is a Radon measure, the mapping $\nu$ defined by 
$$\nu(B):=\int_B \big(|f|+\beta^{-1}g\chi_A\big)^p\,d\mu\qquad(\text{for every Borel set }B\subseteq \Omega)$$
is a Radon measure \cite{folh}. Hence, there are some $0<\epsilon<1$, a function $k\in B(f;r)\cap \Gamma_{g}$ and a compact subset $D$ of $\Omega$ with $\mu(D)>0$ such that
$$\|k-f\|_p<\epsilon^{1/p}\,r.$$
and 
\begin{equation}
\int_{D^c}\big(|f|+\beta^{-1}g\chi_A\big)^p\,d\mu<(1-\epsilon)\,r^p.
\end{equation}
Pick some $\alpha$ with 
$$\|k-f\|_p<\alpha<\epsilon^{1/p}\,r,$$
and denote 
$$\delta:=\frac{\epsilon^{1/p}\,r-\alpha}{2\mu(D)^{\frac{1}{p}}}.$$
Now, we define two functions $h,\xi:\Omega\rightarrow\mathbb C$ by
$$h:=(g\chi_A+\delta)\chi_D+\beta^{-1}g\chi_A\,\chi_{\Omega\setminus D}\qquad\text{and}\qquad \xi:=(|k|+\delta)\eta\,\chi_D+h\chi_{\Omega\setminus D},$$
where 
$$\eta(x):=\left\{
\begin{array}{ll}
	\frac{k(x)}{|k(x)|},& \mbox{ if } k(x)\neq 0\\\\
	1, & \mbox{ if } k(x)=0
\end{array}
\right.$$
for all $x\in\Omega$. Since $D$ is compact, we have $h\chi_A\in L^p(\Omega,\mu)$. Also, for each $x\in D$,
\begin{align*}
|k(x)-\xi(x)|&=\big|k(x)-\big(|k(x)|+\delta\big)\,\eta(x)\big|\\
&=\big|k(x)-k(x)-\delta\,\eta(x)\big|\\
&=\delta,
\end{align*}
and therefore
\begin{align*}
\|(\xi-k)\,\chi_D\|_p=\delta\,\mu(D)^{\frac{1}{p}}=\frac{\epsilon^{1/p}\,r-\alpha}{2}.
\end{align*}
This implies that 
\begin{align*}
\|(\xi-f)\,\chi_D\|_p&\leq \|(\xi-k)\,\chi_D\|_p+\|(k-f)\,\chi_D\|_p\\
&\leq \frac{\epsilon^{1/p}\,r-\alpha}{2}+\alpha<\epsilon^{1/p}\,r.
\end{align*}
Hence, 
\begin{align*}
\|\xi-f\|_p^p&=\int_D|\xi-f|^p\,d\mu+\int_{\Omega\setminus D}|\xi-f|^p\,d\mu\\
&<\epsilon\,r^p+\int_{\Omega\setminus D}|\beta^{-1}g\chi_A-f|^p\,d\mu\\
&\leq \epsilon\,r^p+\int_{\Omega\setminus D}(\beta^{-1}g\chi_A+|f|)^p\,d\mu\\
&<\epsilon\,r^p+(1-\epsilon)\,r^p=r^p,
\end{align*}
and so, $\xi\in B(f;r)$. Moreover,
$$|\xi(x)|=|k(x)|+\delta\geq g(x)+\delta=h(x)\quad\text{a.e. on } D\cap A,$$
and for each $x\in(\Omega\setminus D)\cap A$ we have $|\xi(x)|=h(x)$.
This shows that $\xi\in \Gamma_h$, and so 
$$\varnothing\neq B(f;r)\cap\Gamma_h\subseteq B(f;r)\cap\Gamma_g$$
because $h\geq g$.
Now, let $u\in B(f;r)\cap\Gamma_h$ and put $r':=\min\{\delta,\lambda\,(r-\|f-u\|_p)\}$. Let $v\in B(u;r')$. 
We define the function $\gamma:\Omega\rightarrow\mathbb C$ by
$$\gamma(x):=\left\{
\begin{array}{ll}
v(x), & \mbox{ if } x\in D\\\\
\Big(|v(x)|+\beta|u(x)-v(x)|\Big)\,\theta(x),& \mbox{ if } x\in \Omega\setminus D
\end{array}
\right.$$
where 
$$\theta(x):=\left\{
\begin{array}{ll}
\frac{v(x)}{|v(x)|},& \mbox{ if } v(x)\neq 0\\\\
1, & \mbox{ if } v(x)=0.
\end{array}
\right.$$
Therefore, for each $x\in\Omega\setminus D$ we have 
$$|\gamma(x)-v(x)|=\beta\,|u(x)-v(x)|\qquad\text{and}\qquad|\gamma(x)|\geq \beta\,|u(x)|.$$
Easily,  
\begin{align*}
\|\gamma-v\|_p^p&=\|(\gamma-v)\,\chi_{D}\|_p^p+\|(\gamma-v)\,\chi_{\Omega\setminus D}\|_p^p\\
&=\|(\gamma-v)\,\chi_{\Omega\setminus D}\|_p^p\\
&=\beta^p\,\|(u-v)\,\chi_{\Omega\setminus D}\|_p^p\\
&\leq \beta^p\,\|u-v\|_p^p<\lambda^p\,\|u-v\|_p^p,
\end{align*}
and hence, 
$$\gamma\in B\big(v;\lambda \,\|u-v\|_p\big)\subseteq B(f;r).$$
In addition,
$$|\gamma(x)|\geq \beta\,|u(x)|\geq\beta\,h(x)=g(x)\quad\text{for a.e. }x\in (\Omega\setminus D)\cap A$$
and
$$|\gamma(x)|=|v(x)|\geq |u(x)|-\delta\geq g(x)\quad\text{for a.e. } x\in D\cap A,$$
because $\|u-v\|_p\leq \delta$ and also $|u|\geq h$. Therefore, 
$$B\big(v;\lambda\,\|u-v\|_p\big)\cap B(f;r)\cap \Gamma_g\neq\varnothing.$$
and this competes the proof.
\end{proof}
\begin{remark}
	Note that, in general, the condition \eqref{cond1} in the statement of Theorem \ref{thm1} does not implies that $\Omega$ is a discrete space. In particular, if  in the condition \eqref{cond1} we set $A:=\Omega$, then it implies that $L^p(\Omega,\mu)\subseteq L^\infty(\Omega,\mu)$, and this inclusion is equivalent to
	\begin{equation}\label{alpha}
	\alpha:=\inf\{\mu(E):\, \mu(E)>0\}>0,
	\end{equation}
	and equivalently, for each $q>p$, $L^p(\Omega,\mu)\subseteq L^q(\Omega,\mu)$; see \cite{Villani85}.
	If in addition, ${\rm supp}\mu=\Omega$, then the condition \eqref{alpha} implies that for each $x\in \Omega$,
	$$\mu(\{x\})=\inf\{\mu(F):\,F\text{ is a compact neighborhood of }x\}>0.$$
	Specially, if $\Omega$ is a locally compact group (or hypergroup) and $\mu$ is a left Haar measure of it, then the condition \eqref{cond1} implies that $\Omega$ is a discrete topological space. 
\end{remark}
The next result is a direct conclusion of Theorem \ref{thm1}.
\begin{corollary}\label{corw}
	Let $\Omega$ be a discrete topological space and $\varphi:=(\varphi_j)_{j\in\Omega}\subseteq [1,\infty)$ such that for each $j$, $\varphi_j\geq 1$. Put $\mu_\varphi:=\sum_{j\in\Omega}\varphi_j\,\delta_j$, where $\delta_j$ is the point-mass measure at $j$. Then, for each $g\in L^p(\Omega,\mu_\varphi)$, the set 
	$$\Gamma_{g}:=\big\{f\in L^p(\Omega,\mu_\varphi):\,|f|\geq |g|\big\}$$
	is not $\sigma$-porous in $L^p(\Omega,\mu_\varphi)$. 
\end{corollary}
\begin{proof}
	Just note that for each $k\in\Omega$ and $f\in L^p(\Omega,\mu_\varphi)$,
	$$\|f\|_p^p=\sum_{j\in\Omega}|f(j)|^p\,\mu_\varphi(\{j\})\geq |f(k)|\,\varphi_k\geq |f(k)|^p.$$
\end{proof}
In particular, if a set is endowed with the counting measure, we get the fact.
\begin{corollary}
	Let $p\geq 1$ and $A$ be a non-empty set. Then, for each $g\in \ell^p(A)$, the set 
	$$\Gamma_{g}:=\big\{f\in \ell^p(A):\,|f|\geq |g|\big\}$$
	is not $\sigma$-porous in $\ell^p(A)$. 
\end{corollary}
The situation for $L^\infty$-spaces is different.
\begin{theorem}\label{thm2}
		Let $\Omega$ be a locally compact Hausdorff space and $\mu$ be a nonnegative Radon measure on $\Omega$.	Then, for each $g\in L^\infty(\Omega,\mu)$, the set 
	$$\Gamma_{g}:=\big\{f\in L^\infty(\Omega,\mu):\,|f|\geq |g|\,\,\,{\rm a.e.}\big\}$$
	is not $\sigma$-porous in $L^\infty(\Omega,\mu)$. 
\end{theorem}
\begin{proof}
Same as the proof of Theorem \ref{thm1} fix $0<\lambda\leq \frac{1}{2}$,  and set 
$$\mathcal F:=\big\{\Gamma_{g}:\,\, g\in L^\infty(\Omega,\mu)\big\}.$$
This collection satisfies the conditions of Lemma \ref{lem1}.	Trivially, $\Gamma_g$ is a closed subset of $L^\infty(\Omega,\mu)$ for all $g\in L^\infty(\Omega,\mu)$. Let Assume that $0\leq g\in L^\infty(\Omega,\mu)$, and let $f\in L^\infty(\Omega,\mu)$ and $r>0$. If $B(f;r)\cap \Gamma_g\neq \varnothing$, we choose some $k\in B(f;r)\cap \Gamma_g$ and we find some $\varepsilon\in(0,1)$ such that $\|k-f\|_\infty<\varepsilon r$. Pick some $\delta\in(0,(1-\varepsilon)r)$, and set 
$$h:=g+\delta\qquad\text{and}\qquad \xi:=(|k|+\delta)\eta,$$ 
where $\eta$ is as in the proof of Theorem \ref{thm1}. Then, we get 
	$$\|\xi-f\|_\infty\leq \|\xi-k\|_\infty+\|k-f\|_\infty\leq \delta+\varepsilon r<(1-\varepsilon)r+\varepsilon r=r,$$
	so $\xi\in B(f;r)$, and $\xi\in\Gamma_h$ since $|k|\geq g$. Next, let $u\in B(f;r)\cap \Gamma_h$, and set $r':=\min\{\delta,\lambda\,(r-\|f-u\|_\infty)\}$. Pick some $v\in B(u;r')$. Then, $\|u-v\|_\infty<r'\leq \delta$, so $|v|\geq |u|-\delta\geq h-\delta=g$ a.e., so $v\in\Gamma_g$. Thus, 
	$$v\in B(v;\lambda\|u-v\|_\infty)\cap B(f;r)\cap\Gamma_g.$$
	 This completes the proof.
\end{proof}
\begin{remark}
	The main Theorem \ref{thm1} is valid also for the sequence space $c_0$, because the sequences with finitely many non-zero coefficients approximate sequences in $c_0$. 
\end{remark}
At the end of this section, we give a class of non-$\sigma$-porous subsets of the  $L^p$-space on real line.  In the proof of this result, which is also based on Lemma \ref{lem1}, we apply some functions defined in the proof of Theorem \ref{thm1}.
\begin{theorem}\label{thm3}
	Let $p\geq 1$, and $\tau$ be the Lebesgue measure on $\mathbb{R}$. For each $g\in L^p(\mathbb{R},\tau)$ put 
	$$\Theta_g:=\big\{f\in L^p(\mathbb{R},\tau):\,\|f\chi_{[m,m+1]}\|_p\geq \|g\chi_{[m,m+1]}\|_p\,\text{ for all }m\in\mathbb{Z}\big\}.$$
	Then, $\Theta_g$ is not $\sigma$-porous in $L^p(\mathbb{R},\tau)$.
\end{theorem}
\begin{proof}
	Let $0<\lambda\leq \frac{1}{2}$, and $0<\beta<\lambda$.	Denote
	$$\mathcal F:=\big\{\Theta_{g}:\,\, g\in L^p(\mathbb{R},\tau)\big\}.$$
	We prove that the collection $\mathcal F$ satisfies the conditions of Lemma \ref{lem1}. Let $0\leq g\in L^p(\mathbb{R},\tau)$.  Then, easily $\Theta_{g}\neq \varnothing$ and it is closed in $L^p(\mathbb{R},\tau)$. Now, assume that $f\in L^p(\mathbb{R},\tau)$ and $r>0$ with $B(f;r)\cap \Theta_{g}\neq\varnothing$. Then, there exists a large enough number $N\in\mathbb{N}$, some $0<\epsilon<1$ and a function $k\in B(f;r)\cap \Theta_{g}$ such that
	$$\|k-f\|_p<\epsilon^{\frac{1}{p}}\,r\qquad\text{and}\qquad\int_{[-N,N]^c}\big(|f|+\beta^{-1}g\big)^p\,d\tau<(1-\epsilon)\,r^p.$$
	Pick some $\alpha$ with $\|k-f\|_p<\alpha<\epsilon^{\frac{1}{p}}\,r,$
	and denote 
	$\delta:=\frac{\epsilon^{\frac{1}{p}}\,r-\alpha}{2(2N)^{\frac{1}{p}}}$.
	Put
	$$A_1:=\{m\in[N]:\,g=0 \,\text{ a.e. on }[m,m+1]\},\quad A_2:=[N]\setminus A_1$$
	and 
	$$B_1:=\{m\in[N]:\,k=0 \,\text{ a.e. on }[m,m+1]\},\quad B_2:=[N]\setminus B_1,$$
	where $[N]:=\{-N,\ldots, N-1\}$,
and then define 
$$\rho:=\sum_{m\in A_1}\chi_{[m,m+1]}+\sum_{m\in A_2}\frac{g\chi_{[m,m+1]}}{\|g\chi_{[m,m+1]}\|_p},$$
and 
$$\eta:=\sum_{m\in B_1}\chi_{[m,m+1]}+\sum_{m\in B_2}\frac{k\chi_{[m,m+1]}}{\|k\chi_{[m,m+1]}\|_p}.$$
	Now, we define $h,\xi:\mathbb{R}\rightarrow\mathbb C$ by
	$$h:=g\chi_{[-N,N]}+\delta\rho+\beta^{-1}g\,\chi_{[-N,N]^c}$$
	and 
	$$\xi:=|k|\,\chi_{[-N,N]}+\delta\eta+h\chi_{[-N,N]^c}.$$
Clearly, $h\in L^p(\mathbb{R},\tau)$. 
For each $x\in[-N,N]$ we have $|k(x)-\xi(x)|=\delta\,|\eta(x)|$, and so 
\begin{align*}
\|(k-\xi)\chi_{[-N,N]}\|_p^p&=\delta^p\,\|\eta\chi_{[-N,N]}\|_p^p\\
&=\delta^p\,\sum_{m\in [N]}\|\eta\chi_{[m,m+1]}\|_p^p\\
&=\delta^p\,2N.
\end{align*}
Hence, $\|(k-\xi)\chi_{[-N,N]}\|_p=\delta\,(2N)^{\frac{1}{p}}$.
 Now, similar to the proof of Theorem \ref{thm1} we have $\xi\in B(f;r)$.
Moreover,
$$\|\xi\chi_{[m,m+1]}\|_p=\|k\chi_{[m,m+1]}\|_p+\delta\geq \|g\chi_{[m,m+1]}\|_p+\delta=\|h\chi_{[m,m+1]}\|_p$$
for all $m\in[N]$. And also for each $m\notin[N]$,
	$$\|\xi \chi_{[m,m+1]}\|_p=\|h\chi_{[m,m+1]}\|_p\geq \|g\chi_{[m,m+1]}\|_p.$$
 So,
	$$\xi\in B(f;r)\cap\Theta_h\subseteq B(f;r)\cap\Theta_g.$$
	Now, let $u\in B(f;r)\cap\Theta_h$ and put $r':=\min\{\delta,\lambda\,(r-\|f-u\|_p)\}$. Let $v\in B(u;r')$. 
	We define the function $\gamma:\mathbb{R}\rightarrow\mathbb C$ by
	$$\gamma(x):=\left\{
	\begin{array}{ll}
	v(x), & \mbox{ if } x\in [-N,N]\\\\
	\Big(|v(x)|+\beta|u(x)-v(x)|\Big)\,\theta(x),& \mbox{ if } x\in [-N,N]^c
	\end{array}
	\right.$$
	where 
	$$\theta(x):=\left\{
	\begin{array}{ll}
	\frac{v(x)}{|v(x)|},& \mbox{ if } v(x)\neq 0\\\\
	1, & \mbox{ if } v(x)=0.
	\end{array}
	\right.$$
	Similar to the proof of Theorem \ref{thm1}, we have $\gamma\in B\big(v;\lambda \,\|u-v\|_p\big)$.
Now, for each $m\notin[N]$, 
$$|\gamma|\chi_{(m,m+1)}=(|v|+\beta|u-v|)\chi_{(m,m+1)}\geq \beta |u|\chi_{(m,m+1)}.$$
Hence, 
$$\|\gamma \chi_{[m,m+1]}\|_p\geq \beta\,\|u\chi_{[m,m+1]}\|_p\geq \beta\,\|h\chi_{[m,m+1]}\|_p$$
since $u\in B(f;r)\cap\Theta_h$. However, in this case we have $(m,m+1)\in {[-N,N]}^c$, so $h\chi_{(m,m+1)}=\beta^{-1}g\chi_{(m,m+1)}$. Thus, $\beta\|h\chi_{[m,m+1]}\|_p=\|g\chi_{[m,m+1]}\|_p$. If $m\in[N]$,  we have 
$\gamma\chi_{[m,m+1]}=v\chi_{[m,m+1]}$ because $\gamma\chi_{[-N,N]}=v\chi_{[-N,N]}$ and $[m,m+1]\subseteq {[-N,N]}$. We get 
$$\left|\,\|u\chi_{[m,m+1]}\|_p-\|v\chi_{[m,m+1]}\|_p\right|\leq ||(u-v)\chi_{[m,m+1]}\|_p\leq \|u-v\|_p<\delta$$
because $v\in B(u;r')$, hence 
\begin{align*}
\|\gamma\chi_{[m,m+1]}\|_p&=\|v\chi_{[m,m+1]}\|_p\\
&\geq \|u\chi_{[m,m+1]}\|_p-\delta\\
&\geq \|h\chi_{[m,m+1]}\|_p-\delta\\
&=\|g\chi_{[m,m+1]}\|_p.
\end{align*} 
 Therefore, 
	$$\gamma\in B\big(v;\lambda\,\|u-v\|_p\big)\cap B(f;r)\cap \Theta_g,$$
	and the proof is complete.
\end{proof}
\section{Applications}
In this section, we will apply the results of the previous section, to prove that the set of all non-hypercyclic vectors of some sequences of weighted translation operators is non-$\sigma$-porous.
\begin{definition}
	Let $\mathcal X$ be a Banach space. A sequence $(T_n)_{n\in \mathbb{N}_0}$  of operators in $B(\mathcal X)$ is called {\it hypercyclic} if there is an element $x\in\mathcal X$ (called \emph{hypercyclic vector}) such that the orbit $\{T_n(x):\,n\in\mathbb N_0\}$ is dense in $\mathcal X$. The set of all hypercyclic vectors of a sequence $(T_n)_{n\in \mathbb{N}_0}$ is denoted by $HC((T_n)_{n\in \mathbb{N}_0})$.  An operator $T\in B(\mathcal X)$ is called \emph{hypercyclic} if the sequence $(T^n)_{n\in \mathbb N_0}$ is hypercyclic.
\end{definition} 
Let $G$ be a locally compact group and $a\in G$. Then, for each function $f:G\to\mathbb{C}$ we define $L_af:G\rightarrow\mathbb{C}$ by $L_af(x):=f(a^{-1}x)$ for all $x\in G$. Note that if $p\geq 1$, then the left translation operator  $$L_a:L^p(G)\rightarrow L^p(G),\quad f\mapsto L_af$$
is not hypercyclic because $\|L_a\|\leq 1$. Hypercyclicity of \emph{weigted} translation operators on $L^p(G)$ and regarding an aperiodic element $a$ was studied in \cite{chenchu} (an element $a\in G$ is called \emph{aperiodic} if the closed subgroup of $G$ generated by $a$ is not compact).
\begin{definition}\label{wth}
	Let $G$ be a locally compact group with a left Haar measure $\mu$. Fix $p\geq 1$. We denote $L^p(G):=L^p(G,\mu)$. Assume that $w:G\rightarrow (0,\infty)$ is a bounded measurable function (called a \emph{weight}) and $a\in G$. Then, the weighted translation operator $T_{a,w,p}:L^p(G)\rightarrow L^p(G)$ is defined by 
	$$T_{a,w,p}(f):=w\,L_af,\qquad(f\in L^p(G)).$$
\end{definition}
For each $n\in\mathbb N$ we denote 
$\varphi_n:=w\,L_aw\,\ldots\,L_{a^{n-1}}w,$
where $a^0:=e$, the identity element of $G$.
\begin{theorem}
	Let $p\geq 1$, $G$ be a discrete group and $a\in G$. Let $\mu$ be a left Haar measure on $G$ with $\mu(\{e\})\geq 1$ and $(\gamma_n)_n$ be an unbounded sequence of non-negative integers. Let $w:G\rightarrow(0,\infty)$ be a bounded function such that for some finite nonempty set $F\subseteq G$ and some $N>0$ we have   
	$$a^{\gamma_n}F\cap F=\varnothing\qquad(n\geq N),$$
	and
	$$\beta:=\inf\left\{\prod_{k=1}^{\gamma_n}w(a^kt):\,n\geq N,\,t\in F\right\}>0.$$
	Then, the set 
	$$\Lambda:=\big\{f\in L^p(G,\mu):\,\|T_{a,w,p}^{\gamma_n}f-\chi_{F}\|_p\geq \mu(F)^{\frac{1}{p}}\,\text{ for all } n\geq N\big\}$$
	is non-$\sigma$-porous.
\end{theorem}
\begin{proof}
	Let $\Gamma:=\{f\in L^p(G,\mu):\,|f|\geq \frac{1}{\beta}\chi_F\}$. Then, $\Gamma$ is not $\sigma$-porous in $L^p(G,\mu)$ thanks to Theorem \ref{thm2}. Also, for each $f\in \Gamma$ and $n\geq N$ we have 
	\begin{align*}
	\|T_{a,w,p}^{\gamma_n}f-\chi_F\|_p^p&=\int_G|\prod_{k=1}^nw(a^{-\gamma_n+k}x)\,f(a^{-\gamma_n}x)-\chi_F(x)|^p\,d\mu(x)\\
	&=\int_G|\prod_{k=1}^{\gamma_n}w(a^{k}x)\,f(x)-\chi_F(a^{\gamma_n}x)|^p\,d\mu(x)\\
	&=\int_G|\prod_{k=1}^{\gamma_n}w(a^{k}x)\,f(x)-\chi_{a^{-\gamma_n}F}(x)|^p\,d\mu(x)\\
	&\geq\int_F|\prod_{k=1}^{\gamma_n}w(a^{k}x)\,f(x)-\chi_{a^{-\gamma_n}F}(x)|^p\,d\mu(x)\\
	&=\int_F|\prod_{k=1}^{\gamma_n}w(a^{k}x)\,f(x)|^p\,d\mu(x)\\
	&\geq \int_F|\beta\,\frac{1}{\beta}|^p\,d\mu(x)\\
	&=\mu(F).
	\end{align*}
	This completes the proof.
\end{proof}
\begin{example}
	Let $G$ be the additive group $\mathbb{Z}$ with the counting measure. Let $F$ be a finite non-empty subset of $\mathbb{Z}$. Put $N:=\max\{|j|:j\in F\}$. If $w:=(w_n)_{n\in\mathbb{Z}}\subseteq (0,\infty)$ is a bounded sequence with $w_n\geq 1$ for all $n\geq N$.  Then the required conditions in the previous theorem hold with respect to $F$ and $a:=1$.
\end{example}
The following fact is a direct conclusion of the previous theorem.
\begin{corollary}
	Let $p\geq 1$, $G$ be a discrete group and $a\in G$ with infinite order. Let $\mu$ be the counting measure on $G$ and $(\gamma_n)_n$ be an unbounded sequence of non-negative integers. Let $w:G\rightarrow(0,\infty)$ be a bounded function such that for some $t\in G$,
	$$\inf\left\{\prod_{k=1}^{\gamma_n}w(a^kt):\,n\in\mathbb{N}\right\}>0.$$
	Then, the set 
	$$\big\{f\in L^p(G,\mu):\,\|T_{a,w}^{\gamma_n}f-\chi_{\{t\}}\|_p\geq 1\,\text{ for all } n\big\}$$
	is non-$\sigma$-porous.
\end{corollary}
\begin{theorem}
	Let $p\geq 1$, $G$ be a discrete group, and $a\in G$. Let $\mu$ be a left Haar measure on $G$ with $\mu(\{e\})\geq 1$. Let $(\gamma_n)_n$ be an unbounded sequence of non-negative integers and let $w:G\rightarrow(0,\infty)$ be a bounded function such that   
	$$\inf_{n\in\mathbb{N}}\prod_{k=1}^{\gamma_n}w(a^k)>0.$$
	Then, the set 
	$$\Gamma:=\big\{f\in L^p(G,\mu):\,|f(e)|\,\inf_{n\in\mathbb{N}}\prod_{k=1}^{\gamma_n}w(a^k)\geq 1\big\}$$
	is non-$\sigma$-porous. In particular, 
	 setting $T_n:=T_{a,w,p}^{\gamma_n}$ for all $n$, the set of all non-hypercyclic vectors of the sequence $(T_n)_n$ is not $\sigma$-porous in $L^p(G,\mu)$.
\end{theorem}
\begin{proof}
Since $\mu(\{e\})\geq 1$, applying Theorem \ref{thm1} the set 
	$\Gamma$
	is non-$\sigma$-porous, because 
	$$[\inf_{n\in\mathbb{N}}\prod_{k=1}^{\gamma_n}w(a^k)]^{-1}\,\chi_{\{e\}}\in L^p(G,\mu).$$
	Let $f\in\Gamma$. If $n$ is a  nonnegative integer, then for every $x$ in $G$  we have 
	\begin{equation*}
	\|T_nf\|_p\geq  \big|\varphi_{\gamma_n}(x)\,L_{a^{\gamma_n}}f(x)\big|,
	\end{equation*}
	and so 	setting $x=a^{{\gamma_n}}$ we have  
	\begin{align*}
	\|T_nf\|_p&\geq  \big|\varphi_n(a^{\gamma_n})\,L_{a^{\gamma_n}}f(a^{\gamma_n})\big|\\
	&=\Big[\prod_{k=1}^{\gamma_n}w(a^{k})\Big]\,|f(e)|\\
	&\geq |f(e)|\,\inf_{m\in\mathbb{N}}\prod_{k=1}^{\gamma_m}w(a^k)\geq 1.
	\end{align*}
	This implies that the set $\{T_nf:\,n\in\mathbb{N}\}$ is not dense in $L^p(G,\mu)$, and so $\Gamma$ is a subset of the set of all non-hypercyclic vectors of $T$. This completes the proof. 
\end{proof}
 Now, we recall the definition of hypergroups which are generalizations of locally compact groups; see the monograph \cite{blm} and the basic paper \cite{jew} for more details. In locally compact hypergroups the convolution of two Dirac measures is not necessarily a Dirac measure. Let $K$ be a locally compact Hausdorff space. We denote by $\mathbb{M}(K)$ the space of all regular complex Borel measures on $K$, and by $\delta_x$ the Dirac measure at the point $x$. The support of a measure $\mu\in \mathbb{M}(K)$ is denoted by $\text{supp}(\mu)$.
\begin{definition}
	Suppose that $K$ is a locally compact Hausdorff space, $(\mu,\nu)\mapsto \mu\ast\nu$ is a bilinear positive-continuous mapping from $\mathbb{M}(K)\times \mathbb{M}(K)$ into $\mathbb{M}(K)$ (called {\it convolution}), and $x\mapsto x^-$ is an involutive homeomorphism on $K$ (called {\it involution}) with the following properties:
	\begin{enumerate}
		\item[(i) ]$\mathbb{M}(K)$ with $\ast$ is a complex associative algebra;
		\item[(ii) ]if $x,y\in K$, then $\delta_x\ast\delta_y$ is a probability measure with compact support;
		\item[(iii) ]the mapping $(x,y)\mapsto\text{supp}(\delta_x\ast\delta_y)$ from $K\times K$ into $\textbf{C}(K)$ is continuous, where  $\textbf{C}(K)$ is the set of all non-empty compact subsets of $K$ equipped with Michael topology;
		\item[(iv) ]there exists a (necessarily unique) element $e\in K$ (called identity) such that for all $x\in K$, $\delta_x\ast \delta_e=\delta_e\ast\delta_x=\delta_x$;
		\item[(v) ]for all $x,y\in K$, $e\in \text{supp}(\delta_x\ast\delta_y)$ if and only if $x=y^-$;
	\end{enumerate}
	Then, $K\equiv(K,\ast,^-,e)$ is called a locally compact {\it hypergroup}.
\end{definition}
 A nonzero nonnegative regular Borel measure $m$ on $K$ is called the (left) {\it Haar measure} if for each $x\in K$, $\delta_x\ast m=m$. For each $x,y\in K$ and measurable function  $f:K\rightarrow\mathbb{C}$   we denote
$$f(x*y):=\int_K f\,d(\delta_x * \delta_y),$$
while this integral exists. 


\begin{definition}\label{defl}
Suppose that $a:=(a_n)_{n\in\mathbb{N}_0}$ is a sequence in a hypergroup $K$, and $w$ is a weight function on $K$. For each $n\in \mathbb{N}_0$ we define the bounded linear operator $\Lambda_{n+1}$ on $L^p(K)$ by
$$\Lambda_{n+1} f(x):=w(a_0\ast x)\, w(a_1\ast x)\ldots w(a_{n}\ast x)\,f(a_{n+1}\ast x)\qquad (f\in L^p(K))$$
for all $x\in K$. Also, we assume that $\Lambda_0$ is the identity operator on $L^p(K)$.
\end{definition}
Some linear dynamical properties of this sequence of operators were studied in \cite{kum}. The sequence $\{\Lambda_n\}_n$ is a generalization of the usual powers of a single weighted translation operator on $L^p(G)$, where $G$ is a locally compact group. In fact, any locally compact group $G$ with the mapping
	$$\mu\ast\nu\mapsto\int_G\int_G\delta_{xy}d\mu(x)d\nu(y)\qquad (\mu,\nu\in \mathbb{M}(G))$$
	as convolution, and $x\mapsto x^{-1}$ from $G$ onto $G$ as involution is a locally compact hypergroup.
	Let $\eta:=(a_n)_{n\in\mathbb{N}_0}$ be a sequence in $G$, and $w$ be a weight on $G$. Then for each $f\in L^p(G)$, $n\in\mathbb{N}_0$ and $x\in G$, we have
	$$\Lambda_{n+1} f(x)=w(a_0x)\, w(a_1x)\ldots w(a_{n}x)\, f(a_{n+1}x).$$
	In particular,  let $a\in G$ and for each $n\in\mathbb{N}_0$, put $a_n:=a^{-n}$. Then, $\Lambda_n=T_{a,w,p}^n$ for all $n\in\mathbb{N}$.  In this case, the operator  $T_{a,w,p}$ is hypercyclic if and only if the sequence $(\Lambda_n)_n$ is hypercyclic.

Let $K$ be a discrete hypergroup with the convolution $\ast$ between Radon measures of $K$ and the involution $\cdot^-:K\rightarrow K$. Then, by \cite[Theorem7.1A]{jew}, the measure $\mu$ on $K$ given by 
\begin{equation}\label{hard}
\mu(\{x\}):=\frac{1}{\delta_{x}\ast\delta_{x^-}(\{e\})},\qquad(x\in K)
\end{equation}
is a left Haar measure on $K$.
\begin{proposition}\label{exh}
	 Let $K$ be a discrete hypergroup, $\mu$ be the Haar measure \eqref{hard}, and $p\geq 1$. Then for each $g\in L^p(K,\mu)$, the set 
	$$\big\{f\in L^p(K,\mu):\,|f|\geq |g|\big\}$$
	is not $\sigma$-porous in $L^p(K,\mu)$.
\end{proposition}
\begin{proof}
	Just note that for each $x\in K$ we have 
	$\mu(\{x\})\geq 1$	because 
	$$1=\delta_x\ast\delta_{x^-}(K)\geq \delta_{x}\ast\delta_{x^-}(\{e\}).$$
	Hence, the measure space $(K,\mu)$ satisfies the condition of Corollary \ref{corw}.
\end{proof}
Let $a:=(a_n)_{n\in\mathbb{N}}$ be a sequence in a discrete hypergroup $K$ such that $a_n\neq a_m$ for each $m\neq n$, and let $w:K\rightarrow (0,\infty)$ be bounded. We define $h_{a,w}:K\rightarrow \mathbb{C}$ by 
$$h_{a,w}:=\sum_{n\in\mathbb{N}_0}\frac{1}{w(a_0)w(a_1)\ldots w(a_n)}\chi_{\{a_{n+1}\}}.$$
\begin{theorem}
	Let $p\geq 1$, and $K$ be a discrete hypergroup endowed with the left Haar measure \eqref{hard}. Let $a:=(a_n)_{n\in\mathbb{N}_0}\subseteq K$ with distinct terms, and  $w$ be a weight on $K$ such that $h_{a,w}\in L^p(K)$. Then, the set of all non-hypercyclic vectors of the sequence $(\Lambda_n)_n$ is not $\sigma$-porous.
\end{theorem}
\begin{proof}
	First, thanks to Proposition \ref{exh}, the set 
	$$E:=\Big\{f\in L^p(K):\, |f(a_{n+1})|\geq \frac{1}{w(a_0)w(a_1)\ldots w(a_n)}\,\text{ for all }n\Big\}$$
	is not $\sigma$-porous because it equals to the set 
$\Big\{f\in L^p(K):\, |f|\geq h_{a,w}\Big\}.$
Now, for each $f\in E$,
	\begin{align*}
	\|\Lambda_{n+1}f\|_p&\geq\sup_{x\in K} w(a_0\ast x)\, w(a_1\ast x)\ldots w(a_{n}\ast x)\, |f(a_{n+1}\ast x)|\\
	&\geq w(a_0)\, w(a_1) \ldots w(a_n)\,|f(a_{n+1})|\geq 1
	\end{align*}
	for all $n\in\mathbb{N}_0$. This implies that $0$ does not belong to the closure of $\{\Lambda_nf:\, n\in\mathbb{N}\}$ in $L^p(K)$, and so  $E\subseteq [HC((\Lambda_n)_n)]^c$. This completes the proof.
\end{proof}
Since any group is a hypergroup, we can give the fact below.
\begin{corollary}
	Let $p\geq 1$, and $G$ be a discrete group. Let $a\in G$ be of infinite order, $(\gamma_n)_{n\in\mathbb{N}_0}\subseteq \mathbb{N}$ be with distinct terms and  $w:G\rightarrow(0,\infty)$ be a weight such that 
	$$\big(\frac{1}{w(a^{\gamma_0})w(a^{\gamma_1})\ldots w(a^{\gamma_n})}\big)_n\in\ell^p(G).$$
	 Then, the set of all non-hypercyclic vectors of the sequence $(T_{a,w,p}^{\gamma_n})_n$ is not $\sigma$-porous in $\ell^p(G)$.
\end{corollary}
Now, we can write the next corollary which is a generalization of \cite[Theorem 1]{bay}.
\begin{corollary}
	Let $p\geq 1$, $(\gamma_n)_n\subseteq \mathbb{N}$ be strictly increasing  and  $(w_n)_{n\in\mathbb{Z}}$ be a bounded sequence in $(0,\infty)$ such that 
	$$\big(\frac{1}{w_{\gamma_0}w_{\gamma_1}w_{\gamma_2}\ldots w_{\gamma_n}}\big)_n\in\ell^p(\mathbb{Z}).$$
	Then, the set of all non-hypercyclic vectors of the sequence $(T_n)_n$ is not $\sigma$-porous, where 
	$$(T_{n+1}a)_{k}:=w_{\gamma_0}w_{\gamma_1}w_{\gamma_2}\ldots w_{\gamma_n}a_{k+\gamma_{n+1}}\qquad(k\in\mathbb{N}_0)$$
	for all $a:=(a_j)_j\in \ell^p(\mathbb{Z})$.
\end{corollary}
Applying Theorem \ref{thm2} we can speak regarding some more general situation in the case of $p=\infty$. 
Let $\Omega$ be a locally compact Hausdorff space endowed with a nonnegative Radon measure $\mu$. Let $w:\Omega\rightarrow(0,\infty)$ be a bounded measurable function, and $\alpha:\Omega\rightarrow \Omega$ be a bi-measurable mapping such that $\|f\circ\alpha^{\pm 1}\|_\infty=\|f\|_\infty$ for all $f\in L^\infty(\Omega,\mu)$. Then, we define $T_{\alpha,w,\infty}:L^\infty(\Omega,\mu)\rightarrow L^\infty(\Omega,\mu)$ by 
$$T_{\alpha,w,\infty}(f):=w\,(f\circ\alpha)\qquad(f\in L^\infty(\Omega,\mu)).$$
If $\Omega$ be a locally compact group and $a\in\Omega$, setting $\alpha_a(x):=ax$ for all $x\in\Omega$, we denote $T_{a,w,\infty}:=T_{\alpha_a,w,\infty}$. Note that $\alpha^{-1}$ means the inverse function of $\alpha$, and for each $k\in\mathbb{N}$, $\alpha^{-k}:=(\alpha^{-1})^k$.
\begin{theorem}
	Let $T_{\alpha,w,\infty}$ be the weighted composition operator defined as above and let $\{\gamma_n\}_n\subseteq \mathbb{N}$ be a fixed unbounded sequence. Suppose that there exists a sequence $\{A_n\}_n$ of disjoint subsets of $\Omega$ with $\mu(A_n)>0$ for all $n$ such that 
	$$j_{\alpha,w}:=\sum_{n\in\mathbb{N}}\frac{1}{(w\circ\alpha^{-\gamma_n})\,(w\circ\alpha^{-\gamma_n+1})\ldots (w\circ\alpha^{-1})}\chi_{A_n}\in L^\infty(\Omega,\mu).$$
	Then, the set $\{f\in L^\infty(\Omega,\mu):\,\|T^{\gamma_n}_{\alpha,w,\infty}(f)\|_\infty\geq 1\,\,\text{for all } n\}$ is not $\sigma$-porous. In particular, the set of all non-hypercyclic vectors of the sequence $\{T^{\gamma_n}_{\alpha,w,\infty}\}_n$ is not $\sigma$-porous. 
\end{theorem}
\begin{proof}
	Let $E:=\{f\in L^\infty(\Omega,\mu):\,|f|\geq j_{\alpha,w}\}$. Then, $E$ is not $\sigma$-porous thanks to Theorem \ref{thm2}. For each $f\in E$ and $n\in \mathbb{N}$ we have 
	\begin{align*}
	\|T^{\gamma_n}_{\alpha,w,\infty}(f)\|_\infty&=\|\prod_{k=1}^{\gamma_n}(w\circ\alpha^{\gamma_n-k})\,(f\circ\alpha^{\gamma_n})\|_\infty\\
	&=\|\prod_{k=1}^{\gamma_n}(w\circ\alpha^{-k})\,f\|_\infty\\
	&\geq \|\prod_{k=1}^{\gamma_n}(w\circ\alpha^{-k})\,\chi_{A_n}\,f\|_\infty\\
	&\geq \|\prod_{k=1}^{\gamma_n}(w\circ\alpha^{-k})\,\chi_{A_n}\,j_{\alpha,w}\|_\infty\\
	&=1.
	\end{align*}
	This completes the proof.
\end{proof}
\begin{corollary}
	Let $G$ be a locally compact group and $\mu$ be a left Haar measure on $G$. Let $a\in G$ and $w:G\to(0,\infty)$ be a bounded measurable function. Then, if 
	$$\big(\frac{1}{w(a)w(a^2)\ldots w(a^{n})}\big)_n\in L^\infty(G,\mu),$$
	then the set of all non-hypercyclic vectors of the operator $T_{a,w,\infty}$ on $L^\infty(G,\mu)$ is not $\sigma$-porous.
\end{corollary}
\begin{corollary}
	If $(w_n)_{n\in\mathbb{Z}}$ is a bounded sequence such that 
	$$\big(\frac{1}{w_{1}\ldots w_{n}}\big)_n\in\ell^\infty,$$
	then the set of all non-hypercyclic vectors of the sequence $(T_{\gamma_n,w})_n$ is not $\sigma$-porous in $\ell^\infty$.
\end{corollary}
\begin{theorem}\label{317}
	Let $T_{\alpha,w,\infty}$ be the weighted composition operator on $L^\infty(\Omega,\mu)$ and let $F\subseteq \Omega$ be a Borel set  with $0<\mu(F)<\infty$. Let there exists a constant $N>0$ such that for all $n\geq N$, \begin{equation}\label{cond2}
	\alpha^n(F)\cap F=\varnothing,
		\end{equation}
		 and
	$$\beta:=\inf\{\prod_{k=1}^n(w\circ\alpha^{-k})(t):\,n\geq N,\,t\in F\}\neq 0.$$
	Then, the set $$\{f\in L^\infty(\Omega,\mu):\, \|T^n_{\alpha,w,\infty}f-\chi_F\|_\infty\geq 1\,\text{for all }n\geq N\}$$
	is not $\sigma$-porous in $L^\infty(\Omega,\mu)$.
\end{theorem}
\begin{proof}
	Let $\Gamma:=\{f\in L^\infty(\Omega,\mu):\,|f|\geq \frac{1}{\beta}\chi_F\}$. Then by Theorem \ref{thm2}, $\Gamma$ is not $\sigma$-porous in $L^\infty(\Omega,\mu)$. Also, for each $f\in \Gamma$ we have 
	\begin{align*}
	\|T^n_{\alpha,w,\infty}f-\chi_F\|_\infty&=\|\prod_{k=1}^n(w\circ\alpha^{n-k})\,(f\circ\alpha^n)-\chi_F\|_\infty\\
	&=\|\prod_{k=1}^n(w\circ\alpha^{-k})\,f-\chi_F\circ\alpha^{n}\|_\infty\\
	&=\|\prod_{k=1}^n(w\circ\alpha^{-k})\,f-\chi_{\alpha^n(F)}\|_\infty\\
	&\geq \|\prod_{k=1}^n(w\circ\alpha^{-k})\,f\chi_F-\chi_{\alpha^n(F)}\chi_F\|_\infty\\
	&=\|\prod_{k=1}^n(w\circ\alpha^{-k})\,f\chi_F\|_\infty\\
	&\geq \beta\|f\chi_F\|_\infty\geq \beta\,\frac{1}{\beta}\,\|\chi_F\|_\infty=1.
	\end{align*}
	This completes the proof.
\end{proof}
\begin{example}
	Let $\Omega:=\mathbb{R}$ and $\mu$ be the Lebesgue measure. Put $\alpha(t):=t-1$ for all $t\in\mathbb{R}$ and $F:=[0,1]$. If $w\in C_b(\mathbb{R})$ such that $|w(t)|\geq 1$ for all $t\geq k>0$ and $\inf\{|w(t)|:\,t\in [0,1]\}>0$, then the required conditions in the previous theorem hold with respect to $F$.
\end{example}
With some similar proof, one can prove the next fact without the condition \eqref{cond2}.
\begin{theorem}
	Let $T_{\alpha,w,\infty}$ be the weighted composition operator on $L^\infty(\Omega,\mu)$ and let $F\subseteq \Omega$ be a Borel set  with $0<\mu(F)<\infty$  such that  
	$$\inf\{\prod_{k=1}^n(w\circ\alpha^{-k})(t):\,n\geq N,\,t\in F\}\neq 0.$$
	Then, the set $$\{f\in L^\infty(\Omega,\mu):\, \|T^n_{\alpha,w,\infty}f\|_\infty\geq 1\,\text{for all }n\geq N\}$$
	is not $\sigma$-porous in $L^\infty(\Omega,\mu)$. In particular, the set of all non-hypercyclic vectors of the operator $T_{\alpha,w,\infty}$ is not $\sigma$-porous.
\end{theorem}
In sequel, we find some application for Theorem \ref{thm3} regarding hypercyclicity of shift operators on $L^p(\mathbb{R},\tau)$. 
\begin{theorem}
	Consider the weighted translation operator $T_{\alpha,w}$ on $L^p(\mathbb{R},\tau)$ given by $T_{\alpha,w}f:=w\cdot(f\circ\alpha)$,  where $0<w,w^{-1}\in C_b(\mathbb{R})$ and $\alpha(t)=t+1$. For each $n\in\mathbb{N}$ put $A_n:=[n,n+1]=\alpha^n([0,1])$. Set 
	$$y_{\alpha,w}:=\sum_{n\in\mathbb{N}}\frac{1}{\inf_{t\in A_n}\prod_{k=1}^n(w\circ\alpha^{-k})(t)}\chi_{A_n}$$
	and assume that $y_{\alpha,w}\in L^p(\mathbb{R},\tau)$ (in particular $\inf_{t\in A_n}\prod_{k=1}^n(w\circ\alpha^{-k})(t)>0$ for all $n\in\mathbb{N}$). Then, the set 
	$$\{f\in L^p(\mathbb{R},\tau):\,\|T_{\alpha,w}^n(f)\|_p\geq 1\,\text{for all }n\in\mathbb{N}\}$$
	is not $\sigma$-porous.
\end{theorem}
\begin{proof}
	By Theorem \ref{thm3}, the set 
	$$E:=\{f\in L^p(\mathbb{R},\tau):\,\|f\chi_{A_n}\|_p\geq \|y_{\alpha,w}\chi_{A_n}\|_p\,\text{ for all }n\in\mathbb{N}\}$$
	is not $\sigma$-porous, because it equals to 
	$$\{f\in L^p(\mathbb{R},\tau):\,\|f\chi_{[m,m+1]}\|_p\geq \|y_{\alpha,w}\chi_{[m,m+1]}\|_p\,\text{for all }m\in\mathbb{Z}\},$$
	as $y_{\alpha,w}\chi_{[m,m+1]}=0$ for all $m\in\mathbb Z$ with $m\leq 0$. Now, note that for each $f\in E$ and $n\in\mathbb{N}$, 
	\begin{align*}
	\|T_{\alpha,w}^n(f)\|_p^p&=\int_{\mathbb{R}}\left[\prod_{k=1}^n(w\circ\alpha^{n-k})(t)\right]^p\,|(f\circ\alpha^n)(t)|^p\,d\tau\\
	&=\int_{\mathbb{R}}\left[\prod_{k=1}^n(w\circ\alpha^{-k})(t)\right]^p\,|f(t)|^p\,d\tau\\
	&\geq \int_{A_n}\left[\prod_{k=1}^n(w\circ\alpha^{-k})(t)\right]^p\,|f(t)|^p\,d\tau\\
	&\geq \inf_{t\in A_n}\left[\prod_{k=1}^n(w\circ\alpha^{-k})(t)\right]^p\|y_{\alpha,w}\chi_{A_n}\|_p^p\\
	&=\inf_{t\in A_n}\left[\prod_{k=1}^n(w\circ\alpha^{-k})(t)\right]^p\,\frac{1}{\inf_{t\in A_n}\left[\prod_{k=1}^n(w\circ\alpha^{-k})(t)\right]^p}\,\tau(A_n)=1.
	\end{align*}
\end{proof}
Assume now that there exists some $l\in\mathbb{Z}$ such that 
$$\beta:=\inf\{\prod_{k=1}^n(w\circ\alpha^{-k})(t):\,t\in[l,l+1],\,n\in\mathbb{N}\}>0.$$
Put
$$F:=\{f\in L^p(\mathbb{R},\tau):\,\|f\chi_{[m,m+1]}\|_p\geq \|\frac{1}{\beta}\chi_{[l,l+1]}\chi_{[m,m+1]}\|_p\,\text{ for all }m\in\mathbb{Z}\}.$$
So by Theorem \ref{thm3}, $F$ is not $\sigma$-porous. For every $f\in F$ we have 
\begin{align*}
\|T_{\alpha,w}^n(f)\|_p^p&=\int_{\mathbb{R}}\left[\prod_{k=1}^n(w\circ\alpha^{n-k})(t)\right]^p\,|(f\circ\alpha^n)(t)|^p\,d\tau\\
&=\int_{\mathbb{R}}\left[\prod_{k=1}^n(w\circ\alpha^{-k})(t)\right]^p\,|f(t)|^p\,d\tau\\
&\geq \int_{[l,l+1]}\left[\prod_{k=1}^n(w\circ\alpha^{-k})(t)\right]^p\,|f(t)|^p\,d\tau\\
&\geq 1.
\end{align*}
Hence, the set 
$$\{f\in L^p(\mathbb{R},\tau):\,\|T_{\alpha,w}^n(f)\|_p\geq 1\,\text{for all }n\in\mathbb{N}\}$$
is not $\sigma$-porous.

Next, suppose that $\alpha$ is an aperiodic function on $\mathbb{R}$ (this means that for each compact set $C\subset \mathbb{R}$, there exists a constant $N>0$ such that $\alpha^n(C)\cap C=\varnothing$ for all $n\geq N$) and $\beta>0$, where $\beta$ is as above. Then, the set 
$$\{f\in L^p(\mathbb{R},\tau):\, \|T_{\alpha,w}^n(f)-\chi_{[l,l+1]}\|_p\geq 1\,\text{ for all }n\geq N\}$$
is not $\sigma$-porous. Indeed, for all $f\in F$, and $n\geq N$ we have 
$$\|T_{\alpha,w}^n(f)-\chi_{[l,l+1]}\|_p\geq \|\prod_{k=1}^n\,(w\circ\alpha^{-k})\,f\chi_{[l,l+1]}\|_p$$
by the similar calculations as in the proof of Theorem \ref{317}. However, 
$$\|\prod_{k=1}^n\,(w\circ\alpha^{-k})\,f\chi_{[l,l+1]}\|_p\geq \beta\,\|f\chi_{[l,l+1]}\|_p\geq 1.$$
\bibliographystyle{amsplain}

\end{document}